\documentclass[amstex,12pt,reqno]{amsart}
\usepackage{amsmath, amssymb, amsthm}
\usepackage[dvipdfmx]{graphicx}
\usepackage{mathrsfs}

\newtheorem{theorem}{Theorem}
\newtheorem{lemma}[theorem]{Lemma}
\newtheorem{proposition}[theorem]{Proposition}
\newtheorem{corollary}[theorem]{Corollary}
\newtheorem{claim}{Claim}

\theoremstyle{definition}

\newtheorem*{remark}{Remark}

\newtheorem*{question}{Question}

\title
[BMO-Teichm\"uller theory]{A real-variable construction with applications to BMO-Teichm\"uller theory}

\author[H. Wei]{Huaying Wei} 
\address{Department of Mathematics and Statistics, Jiangsu Normal University \endgraf Xuzhou 221116, PR China} 

\email{hywei@jsnu.edu.cn} 

\author[M. Zinsmeister]{Michel Zinsmeister}
\address{Institut Denis Poisson, Universit\'e  d' Orl\'eans, 45067 Orl\'eans Cedex 2, France}
\email{zins@univ-orleans.fr}


\begin{document}

\maketitle

\begin{abstract}
With the use of real-variable techniques, we construct a weight function $\omega$ on the interval $[0, 2\pi)$  that is doubling and satisfies $\log \omega$ is a BMO function, but which is not a  Muckenhoupt weight ($A_\infty$). Applications to the BMO-Teichm\"uller space and the space of chord-arc curves are considered. 
\end{abstract}

\section{Introduction}
Let $\Gamma$ be a bounded Jordan curve in the extended complex plane $\hat{\mathbb C}$. We can consider three objects associated to $\Gamma$: the Riemann mapping $\Phi$ from the unit disk $\mathbb D$ onto the bounded component $\Omega$ of $\hat{\mathbb C}\backslash\Gamma$; the Riemann mapping $\Psi$ from the exterior of the unit disk $\mathbb D^*$ onto the unbounded component $\Omega^*$ of $\hat{\mathbb C}\backslash\Gamma$; the conformal welding corresponding to $\Gamma$, $h = \Psi^{-1}\circ\Phi$, which is a sense-preserving homeomorphism of the unit circle $\mathbb S$. Let $S(\mathbb D)$ be the set of all mappings $\log f'(z)$ where $f$ is conformal (i.e. holomorphic and injective) in $\mathbb D$. By the Koebe distortion theorem, $S(\mathbb D)$ is a bounded subset of the Bloch space $\mathcal B(\mathbb D)$ which consists of holomorphic functions $\varphi$ in $\mathbb D$ with finite norm $\Vert\varphi\Vert_{\mathcal B} = \sup_{z\in\mathbb D}(1 - |z|^2)|\varphi'(z)|$. In particular $\log\Phi' \in S(\mathbb D)$ and $\log\Psi' \in S(\mathbb D^*)$, which is defined in a similar way. 

A bounded Jordan curve $\Gamma$ is called quasicircle if there exists a constant $C > 0$ such that
$$
 {\text diam}(\gamma) \leq C|z - \zeta|
$$
for any $z, \zeta \in \Gamma$, 
where $\gamma$ is the smaller of the two subarcs of $\Gamma$ joining $z$ and $\zeta$. The quasicircle $\Gamma$ can be characterized from the viewpoint of the universal Teichm\"uller space in the following equivalent ways (see \cite{Ah, Le, Na}): 
\begin{enumerate}
\item[(a)]
$\log\Phi'$ belongs to the interior of $S(\mathbb D)$ in Bloch space $\mathcal B(\mathbb D)$.
\item[(b)]
$\log\Psi'$ belongs to the interior of $S(\mathbb D^*)$ in Bloch space $\mathcal B(\mathbb D^*)$.
\item[(c)]
$h$ is a quasisymmetric homeomorphism of $\mathbb S$, namely,  there exists a constant $C > 0$ such that for any adjacent intervals $I, I^* \subset \mathbb S$ of length $|I| = |I^*| \leq \pi$, we have $C^{-1}|h(I)| \leq  |h(I^*)| \leq C|h(I)|$,  where $|\cdot |$ denotes the Lebesgue measure. The optimal value of such $C$ is called the doubling constant for $h$. 
\end{enumerate}
It is well known that a quasisymmetric homeomorphism need not be absolutely continuous, and may be totally singular.  If however it is absolutely continuous, we say $|h'|$ is a doubling weight.

There are analogs of the above statements in the setting of the BMO Teichm\"uller theory,  introduced by Astala and Zinsmeister \cite{AZ}, and investigated in depth later by Fefferman, Kenig and Pipher \cite{FKP}, Bishop and Jones \cite{BJ}, Cui and Zinsmeister \cite{CZ}, Shen and Wei \cite{SW}.  

Let $\Gamma$ be  a bounded quasicircle. Then the following three statements are equivalent:
\begin{enumerate}
\item
$\log\Phi' \in \rm BMOA(\mathbb D)$, the space of analytic functions in $\mathbb D$ of bounded mean oscillation. 
\item
$\log\Psi' \in \rm BMOA(\mathbb D^*)$, the space of analytic functions in $\mathbb D^*$ of bounded mean oscillation. 
\item
$h$ is a strongly quasisymmetric homeomorphism of $\mathbb S$, namely, for each $\epsilon > 0$ there is a $\delta > 0$ such that 
$$
|E| \leq \delta |I| \Rightarrow |h(E)| \leq \epsilon |h(I)|
$$
whenever $I \subset \mathbb S$ is an interval and $E \subset I$ a measurable subset: we say that $h$ is absolutely continuous and $|h'|$ is an $A_{\infty}$-weight (or Muckenhoupt weight). 
\end{enumerate}
The set of strongly quasisymmetric homeomorphisms of $\mathbb S$ is a group; more precisely, it is the group of homeomorphisms $h$ such that $P_h:\, b\mapsto b\circ h$ is an isomorphism of the BMO space ${\rm BMO}(\mathbb S)$ (see \cite{Jo}).  
Naturally an $A_{\infty}$-weight is  doubling. Fefferman and Muckenhoupt  \cite{FM} gave this a direct computation, and they also provided an example of a function that satisfies the doubling condition but not $A_{\infty}$. 

Noting that $h = \Psi^{-1}\circ\Phi$, we conclude that
$$
\log h' = \log\Phi' - \log\Psi'\circ h.
$$
If one of the above three characterizations is true, then it holds that 
\begin{enumerate}
\item[(4)]
$h$ is absolutely continuous and $\log h' \in \rm BMO(\mathbb S)$.
\end{enumerate}
Recall that an integrable function $u$ on  $\mathbb S$  is said to have bounded mean oscillation, i.e.  $u \in \rm BMO(\mathbb S)$,  if
$$
\Vert u \Vert_* = \sup_{I \subset \mathbb S}\frac{1}{|I|} \int_I |u(z)-u_I| |dz| <\infty,
$$
where the supremum is taken over all bounded intervals $I$ on $\mathbb S$ and $u_I$ denotes the integral mean of $u$
over $I$. This is regarded as a Banach space with norm $\Vert \cdot \Vert_*$  modulo constants since obviously constant functions have norm zero. 
An integrable function $u$ on $\mathbb S$ is said to have vanishing mean oscillation, i.e.  $u \in {\rm VMO}(\mathbb S)$,  if $\Vert u \Vert_* < \infty$ and moreover 
$$ 
\lim_{|I| \to 0}\frac{1}{|I|} \int_I |u(z)-u_I| |dz|=0. 
$$
This is a closed subspace of ${\rm BMO}(\mathbb S)$, actually the closure of the space of all continuous functions on $\mathbb S$ under the norm $\Vert \cdot \Vert_*$.  
If $\log h' \in \rm BMO(\mathbb S)$ with a small norm, or if $\log h' \in \rm VMO(\mathbb S)$, then it can be checked easily that  $|h'|$ is an $A_{\infty}$-weight by the John-Nirenberg inequality (see \cite{Ga})(see also \cite[Proposition 5.4]{WM} for a proof). 

In the present paper, in section 2 we will construct an example of a  weight function with the use of real-variable techniques that shows $(4) \nRightarrow (3)$ in the premise of $\Gamma$ being a bounded quasicircle. 
Besides that, this construction implies more. More precisely, we will prove the followings in sections 3, 4, 5, respectively. 

\begin{theorem}
There exists a sense-preserving homeomorphism $h$ of $\mathbb S$ such that $h$ is absolutely continuous,  $|h'|$ is a doubling weight, $\log h' \in \rm BMO(\mathbb S)$, but $|h'|$ is not an $A_{\infty}$-weight. 
\end{theorem}

\begin{corollary}
There exist a sequence $\{\Gamma_t\}$ $(0  \leq t \leq 1)$ of quasicircles and a constant $C > 0$ such that $\Vert \log h_t' \Vert_* \leq C$ for any $0 \leq t \leq 1$, $\log \Phi_t' \in \rm BMOA(\mathbb D)$ for any $0 \leq t < 1$,  but $\Vert \log\Phi_t'\Vert_* \to \infty$ as $t \to 1$. 
\end{corollary}

\begin{theorem} 
For any $\epsilon > 0$ there exists a rectifiable quasicircle $\hat\Gamma$ for which $\log \hat\Phi'\in  \mathrm{BMOA}(\mathbb D) $ with $\Vert \log\hat\Phi'\Vert_{\mathcal B}<\epsilon$  such  that $\hat\Gamma$ is not a chord-arc curve. 
\end{theorem}

\section{A ``real-variable" construction}
 For $0 < \epsilon < 1$, set
$$
P_n(t) = \prod_{j=0}^{n} (1 + \epsilon \cos (3^j t)), \qquad t \in [0, 2\pi). 
$$
This is a trigonometric polynomial, and it is known that $P_n(t) \geq0$ and $\int_0^{2\pi}P_n(t) dt = 2\pi$.  Set $\mu_n(x) = \int_0^x P_n(t) dt$ so that $d \mu_n(x) = P_n(x) dx$.  Here, $P_n(x)$ and $\mu_n(x)$ can also be defined for $x \in \mathbb R$ by the periodic extension 
$P_n(x + 2\pi) = P_n(x)$ and by the condition $\mu_n(x + 2\pi) - \mu_n(x) = 2\pi$.  The sequence $\{\mu_n\}$ converges to a non-decreasing limit function $\mu$, 
singular with respect to the Lebesgue measure on $[0, 2\pi)$ (see \cite[Vol. I, Theorem 7.6]{Zy}), which implies that the corresponding Lebesgue-Stieltjes measure sequence 
$\{d \mu_n\}$ converges weakly to a measure $d\mu$ of total mass $2\pi$. A direct proof of this fact was given in \cite[P. 125]{Zi}. 
We recall that this means convergence in the following sense:
\begin{equation}\label{weak}
\int_0^{2\pi} \varphi (x) d \mu_n(x) \to \int_0^{2\pi} \varphi (x) d \mu(x), \qquad n \to \infty
\end{equation}
for any $\varphi \in C_b([0, 2\pi))$, the set of continuous and bounded functions on $[0, 2\pi)$.  The Riesz product
\begin{equation*}
P(t) = \prod_{j=0}^{\infty} (1 + \epsilon \cos (3^j t)), \qquad t \in [0, 2\pi)
\end{equation*}
may then be expanded as a well-defined trigonometric series, which is the Fourier series of the measure $d\mu$ (also called the Fourier-Stieltjes series of the function $\mu$ in the literature). 

\begin{claim}\label{infty}
For any $p > 1$, it holds that 
$$
\lim_{n \to \infty} \Vert P_n \Vert_{p} = \infty. 
$$
\end{claim}
\begin{proof}
 Assuming that there exists a positive constant $C$ and a subsequence $\{n_k\}$ such that 
$$
\Vert P_{n_k} \Vert_{p} \leq C
$$
for any $k$, then,  by the Banach-Alaoglu theorem, the sequence $\{P_{n_k}\}$ has a weak-star convergent subsequence $\{P_{n'_k}\}$, converging to some $\hat{P} \in L^p([0, 2\pi))$, namely, 
\begin{equation}\label{q}
\int_0^{2\pi} \varphi (x) P_{n_k'}(x) dx \to \int_0^{2\pi} \varphi (x) \hat{P}(x) dx, \qquad k \to \infty
\end{equation}
for any $\varphi \in L^q([0, 2\pi))$. In particular, taking $\varphi \in C_b([0, 2\pi))$, we conclude by (\ref{weak}) and (\ref{q}) that $d\mu(x) = \hat P(x) dx$, which contradicts the fact that $d\mu$ is singular with respect to the Lebesgue measure on $[0, 2\pi)$. 
\end{proof}

Set $p_n = 1+ 1/n$ so that $p_n \to 1$ as $n \to \infty$.  By Claim \ref{infty}, for any $n \geq 1$ there exists an integer $N_n$ such that  
\begin{equation}\label{4}
\Vert P_{N_n} \Vert_{p_n} \geq 4^n. 
\end{equation}
Set
$$
\tilde f(x) = \sum_{n \geq 1} 2^{-n} \frac{P_{N_n}}{\Vert P_{N_n} \Vert_{L\log L}}, \quad x \in [0, 2\pi).
$$
Here, $\Vert g \Vert_{L\log L} = \int_{I} |g(x)|\log(e+|g(x)|) dx$ for any  integrable function $g$ on the bounded interval $I$, and we say that $g \in L\log L(I)$ if $\Vert g \Vert_{L\log L} < \infty$. 
Set
$$
\omega(x) = M\tilde f (x), \qquad x \in [0, 2\pi).
$$
Here, $M\tilde f$ is the Hardy-Littlewood maximal function of $\tilde f$.  For any $t \in [0, 1]$, set $h_t(e^{ix}) = e^{i g_t(x)}$ by
$$
g_t(x) = \int_0^x \omega^t(s) ds, \quad x \in [0, 2\pi).
$$
Then, $|h_t'(e^{ix})| = \omega^t(x)$ and 
$$
\log h_t'(e^{ix}) = \log g_t'(x) + i(g_t(x) - x)
$$
whose imaginary part is clearly a continuous function, and in particular a BMO function on the interval $[0, 2\pi)$.
Let $\Gamma_t$ be the bounded Jordan curve whose conformal welding is $h_t$, and $\Phi_t$, $\Psi_t$ the Riemann mappings associated to the two components $\Omega_t$ and $\Omega^*_t$ of $\hat{\mathbb C}\backslash\Gamma_t$, respectively. We will denote $h_1$ by $h$ for simplicity in the following. 
  
\section{Proof of Theorem 1}
In this section, we prove the function $h$ above is a desired one for Theorem 1. Before that, 
We need to show that $\tilde f \notin L^p([0, 2\pi))$ for any $p > 1$ (see Claim \ref{notin}), but $\tilde f \in L \log L([0, 2\pi))$ (see Claim \ref{LlogL}). For this purpose we first  recall some well-known facts. 

Let $f$ be a measurable function on a measure space $(X, \nu)$. The 
distribution function
$$
m(t) = \nu(\{x \in X:\; |f(x)| > t\})
$$
defined for $t > 0$  is a decreasing function of $t$, and it determines the $L^p$ norms of $f$. If $f \in L^{\infty}$ then $\Vert f \Vert_{\infty} = \sup \{t: \; m(t) > 0\}$, and if $f \in L^p$ $(0 < p < \infty)$  then the Chebychev inequality says
\begin{equation}\label{Che}
m(t) \leq \frac{1}{t^p} \int_{|f| \geq t} |f|^p d\nu, 
\end{equation}
and in particular $m(t) \leq \Vert f \Vert_{p}^p/t^p$. 

\begin{lemma}\label{Fubini}
Let $\psi : [0, \infty) \to [0, \infty)$ be an increasing differentiable function such that $\psi (0) = 0$. If $f(x)$ is a non-negative measurable function in a measure space $(X, \nu)$, then
\begin{equation}\label{Fu}
\int_X \psi(f(x)) d\nu = \int_0^{\infty} \psi'(x) m(t) dt. 
\end{equation}
\end{lemma}

\begin{proof}
We may assume $f$ vanishes except on a set of $\sigma$-finite measure because otherwise both sides of (\ref{Fu}) are infinite. Then the Fubini theorem shows that both sides of (\ref{Fu}) equal the product measure of the ordinate set $\{(x, t): \; 0 < t < \psi(f(x))\}$. That is, 
\begin{equation*}
\begin{split}
\int_X\psi(f(x)) d\nu & = \int_X\int_0^{f(x)} \psi'(t) dt d\nu = \int_0^{\infty} \psi'(t) \nu(\{f > t\}) dt\\
& = \int_0^{\infty} \psi'(t) m(t) dt.
\end{split}
\end{equation*}
\end{proof}

\begin{proposition}\label{p-1}
Let $f$ be a non-negative measurable function in a measure space $(X, \nu)$ with norm $\Vert f \Vert_1 = 1$. If  $\Vert f \Vert_p \geq 2$ for some $p > 1$, then there is a positive constant $C$ depending only on $p$ such that 
$$
\int_X f(x) \log (e + f(x)) d\nu \leq \frac{C}{(p - 1)^2} \log \Vert f \Vert_p.
$$
\end{proposition}
\begin{proof}
By Lemma \ref{Fubini} and taking $\psi (t) = t\log (e+ t)$ we have
\begin{equation}\label{applyF}
\int_X f(x) \log(e + f(x)) d\nu = \int_0^{\infty} \psi'(t) m(t) dt < 2 \int_0^{\infty} \log (e + t) m(t) dt.
\end{equation}
Let now $T > 0$ to be determined later. We divide the right integral of the above inequality by $T$ into two parts and then estimate them respectively:
$$
\int_0^{\infty} \log (e + t) m(t) dt = \int_0^{T} \log (e + t) m(t) dt + \int_T^{\infty} \log (e + t) m(t) dt.
$$
By Lemma \ref{Fubini} and $\Vert f \Vert_1 = 1$ we have
\begin{equation*}
\begin{split}
 \int_0^{T} \log (e + t) m(t) dt &\leq \log(e+T)\int_0^T m(t) dt\\
 & \leq \log(e+T)\int_X f(x) d\nu \leq \log (e+T). 
 \end{split}
\end{equation*}
By the Chebychev inequality we have
\begin{equation*}
\begin{split}
\int_T^{\infty} \log (e + t) m(t) dt &\leq \Vert f \Vert_p^p \int_T^{\infty} \frac{\log (e + t)}{t^p} dt\\
& < \frac{\Vert f \Vert_p^p}{(p-1)^2 T^{p-1}} ( (p-1)\log(e+T) + 1). 
\end{split}
\end{equation*}
Then, by  choosing  $T = \Vert f \Vert_p^{\frac{p}{p-1}}$ and substituting what have obtained into (\ref{applyF}), it follows from $\Vert f \Vert_{p} \geq 2$ that
$$
\int_X f(x) \log(e + f(x)) d\nu \leq \frac{C}{(p - 1)^2} \log \Vert f \Vert_p.
$$
\end{proof}

\begin{claim}\label{notin}
For any $p > 1$, $\tilde f$ is not in $L^p([0, 2\pi))$. 
\end{claim}

\begin{proof}
It follows from Proposition \ref{p-1} and (\ref{4}) that
\begin{equation*}
\begin{split}
\Vert \tilde f  \Vert_{p_n} & > \frac{\Vert  P_{N_n} \Vert_{p_n}}{2^n\Vert P_{N_n} \Vert_{L\log L}} \geq 
C^{-1}(p_n - 1)^2 \frac{\Vert  P_{N_n} \Vert_{p_n}}{2^n\log (\Vert P_{N_n} \Vert_{p_n} )} > \left(\frac{3}{2}\right)^{n}
\end{split}
\end{equation*}
as $n$ is sufficiently large.  

Assuming $\tilde f \in L^p([0, 2\pi))$ for some $p > 1$ with $\Vert \tilde f \Vert_{p}$ being a constant $C_1$. Then, 
$\tilde f \in L^q([0, 2\pi))$ for any $1 < q < p$ and $\Vert \tilde f \Vert_{q} \leq C_1$. However, taking some integer $n$ so that $\left(\frac{3}{2}\right)^n > C_1$ and $p_n < p$, we then have $\Vert \tilde f \Vert_{p_n} > C_1$. This leads to a contradiction. 
\end{proof}

\begin{claim}\label{LlogL}
It holds that
$$
\tilde f(x) \log (e+\tilde f(x)) \leq \sum_{n \geq 1} 2^{-n} \frac{P_{N_n}(x)\log (e+ P_{N_n}(x))}{\Vert P_{N_n} \Vert_{L \log L}}, 
$$
and moreover  $\Vert \tilde f \Vert_{L \log L} \leq 1$. 
\end{claim}

\begin{proof}
Since $\psi(t) = t\log (e+ t)$ is convex on  $\mathbb R^+$, we have 
\begin{equation*}
\begin{split}
\tilde f(x) \log (e+\tilde f(x)) & = \psi\circ \tilde f(x) = \psi \left(\frac{\sum_{n \geq 1}2^{-n}\frac{P_{N_n}(x)}{\Vert P_{N_n} \Vert_{L \log L}}}{\sum_{n\geq 1}2^{-n}}  \right)\\
&\leq \sum_{n \geq 1} 2^{-n} \frac{1}{\Vert P_{N_n} \Vert_{L \log L}} P_{N_n}(x) \log \left(e + \left(\frac{P_{N_n}(x)}{\Vert P_{N_n} \Vert_{L \log L}}  \right)\right). 
\end{split}
\end{equation*}
Then,  
$$
\tilde f(x) \log (e +\tilde f(x)) \leq \sum_{n \geq 1} 2^{-n} \frac{P_{N_n}(x) \log (e +P_{N_n}(x))}{\Vert P_{N_n}\Vert_{L \log L}}
$$
and thus $\Vert \tilde f \Vert_{L \log L} \leq 1$.
\end{proof}

The Hardy-Littlewood maximal function $M\nu$ of the signed measure $\nu$ is defined as 
$$M\nu(x) = \sup_{x \in I}\frac{1}{|I|}|\nu|(I),$$
where the supremum is taken over all bounded intervals. In particular, for the signed measure of the form $ d\nu(x) = g(x)dx$, $M\nu$ is usually denoted by $Mg$ and called the Hardy-Littlewood maximal function of the function $g$ in the literature. This is a quantitation of the Lebesgue theorem which says that if $g(x)$ is locally integrable on $\mathbb R$ then
$$
\lim_{h, k \to 0^+} \frac{1}{h+k} \int_{x - h}^{x+k} g(t) dt = g(x)
$$
for almost every $x \in \mathbb R$. If $g \in L^p(\mathbb R)$ for $p \in [1, \infty]$, then $Mg(x)$ is finite almost everywhere. 
Moreover, $Mg \in L^p(\mathbb R)$ if $g \in L^p(\mathbb R)$ for $p \in (1, \infty]$, while 
 $Mg$ is  weak $L^1$ if $g \in L^1(\mathbb R)$ (see \cite[Page 23]{Ga}). On the other hand, Stein \cite{St} and Zygmund \cite{Zy} proved that if $g$ is supported on a finite interval $I$  then $Mg \in L^1(I)$ if and only if $g \in L\log L(I)$. 

\begin{proof}[Proof of Theorem 1]
Recall that 
$$
\omega(x) = M\tilde f (x), \qquad x \in [0, 2\pi).
$$
Then, we have  that $\omega$ is finite almost everywhere, $\omega \in L^1([0, 2\pi))$  by Claim \ref{LlogL}, and   $\omega \notin L^p([0, 2\pi))$ for any $p > 1$ by Claim \ref{notin} since it holds that $\omega(x) \geq \tilde f(x)$ from the Lebesgue theorem, and thus $\omega$ is not an $A_{\infty}$-weight. Indeed, if  $\omega$ were an $A_{\infty}$-weight, then the reverse H\"older inequality holds for $\omega$ (see \cite{Ga}), namely, there are $\delta > 0$ and $C > 0$ such that
$$
\left( \frac{1}{|I|} \int_I \omega (x)^{1 + \delta} dx\right)^{1/(1+\delta)} \leq \frac{C}{|I|}\int_I \omega (x) dx
$$
for any interval $I \subset [0, 2\pi)$. 
This contradicts that $\omega \notin L^p([0, 2\pi))$ for any $p > 1$.

Furthermore, we note that the claim $\log \omega \in \rm BMO([0, 2\pi))$ follows from a result by Coifman and Rochberg \cite{CR}: assuming $\nu$ is a locally finite signed  Borel measure on $\mathbb R$ for which the maximal function $M\nu(x)$ is finite almost everywhere we have $\log M\nu \in \rm BMO(\mathbb R)$.  Thus,  by taking $\nu(x) = \tilde f(x)dx$ we get the claim $\log \omega = \log M\nu \in \rm BMO([0, 2\pi))$. 

It remains to show that $\omega$ is a doubling weight on $[0, 2\pi)$. It is sufficient to check the doubling condition
$\int_I \omega (x) dx \asymp \int_{I^*}\omega (x) dx$ holds when $I$ and $I^*$ are two adjacent intervals of length $\frac{2\pi}{3^n}$. 

We cut the sum giving $\tilde f$ into two parts, each term in the first part with the subscript $N_k$ such that $N_k \leq n-1$. Then we can write
$$
\tilde f(x) = g_n(x) + f_n(x)h_n(x).
$$
Here, we split the second part into the product of $f_n$ and $h_n$, $f_n$ with the subscript $N_k$'s such that $N_k \leq n-1$, and $h_n$ being $\frac{2\pi}{3^n}$ periodic. Then, 
$$
\exp \left(-\frac{\pi\epsilon}{1 - \epsilon}\right) f_n(x) \leq f_n(x + \frac{2\pi}{3^n}) \leq \exp \left(\frac{\pi\epsilon}{1 - \epsilon}\right) f_n(x)
$$
and
$$
\exp \left(-\frac{\pi\epsilon}{1 - \epsilon}\right) g_n(x) \leq g_n(x + \frac{2\pi}{3^n}) \leq \exp \left(\frac{\pi\epsilon}{1 - \epsilon}\right) g_n(x).
$$
It follows that
\begin{equation}\label{tilde}
\exp \left(-\frac{\pi\epsilon}{1 - \epsilon}\right) \tilde f(x) \leq \tilde f(x + \frac{2\pi}{3^n}) \leq \exp \left(\frac{\pi\epsilon}{1 - \epsilon}\right) \tilde f(x). 
\end{equation}
Indeed, for any $k \leq n-1$, we write $P_k(x) = \prod_{j=1}^{k}(1 + \epsilon \cos(3^jx)) = \prod_{j=1}^{k}\varphi_j(x)$ so that $\log P_k(x) = \sum_{j=1}^{k}\log \varphi_j(x)$, and then by using the finite increment theorem we get 
\begin{equation*}
\begin{split}
|\log P_k(x + \frac{2\pi}{3^n}) - \log P_k(x)| &\leq \sum_{j=1}^{k} |\log\varphi_j(x + \frac{2\pi}{3^n}) - \log \varphi_j(x)|\\
& \leq \sum_{j=1}^{k}\Vert (\log \varphi_j(x))' \Vert_{\infty} \frac{2\pi}{3^n}\\
& \leq \frac{2\pi\epsilon}{1 - \epsilon} \sum_{j=1}^{k} 3^{j-n} \leq \frac{\pi\epsilon}{1-\epsilon}, 
\end{split}
\end{equation*}
which implies 
$$
\exp \left(-\frac{\pi\epsilon}{1 - \epsilon}\right) P_k(x) \leq P_k(x + \frac{2\pi}{3^n}) \leq \exp \left(\frac{\pi\epsilon}{1 - \epsilon}\right) P_k(x).
$$

For any $x \in [0, 2\pi)$ and any interval $I = [a, b]$ with $a \leq x \leq b$, set $J = [a + \frac{2\pi}{3^n}, b + \frac{2\pi}{3^n}]$ so that $x + \frac{2\pi}{3^n} \in J$. This gives
$$
\int_J \tilde f(t)dt = \int_I \tilde f(t + \frac{2\pi}{3^n}) dt. 
$$
Combined with (\ref{tilde}), it implies
$$
\exp \left(-\frac{\pi\epsilon}{1 - \epsilon}\right)\left(\frac{1}{|I|} \int_I \tilde f(t)dt \right) \leq \frac{1}{|J|} \int_J \tilde f(t)dt \leq \exp \left(\frac{\pi\epsilon}{1 - \epsilon}\right)\left(\frac{1}{|I|} \int_I \tilde f(t)dt \right), 
$$
and then by taking the supremum over the interval $I$ containing $x$ we pass (\ref{tilde}) to the maximal function $\omega (x)$, namely, 
\begin{equation}\label{omega}
\exp \left(-\frac{\pi\epsilon}{1 - \epsilon}\right) \omega(x) \leq \omega(x + \frac{2\pi}{3^n}) \leq \exp \left(\frac{\pi\epsilon}{1 - \epsilon}\right) \omega(x). 
\end{equation}
For any two adjacent intervals $I$ and $I^*$ of length $\frac{2\pi}{3^n}$, we have
$$\int_{I^*}\omega (x) dx = \int_I \omega(x + \frac{2\pi}{3^n}) dx.$$
By combining this with (\ref{omega}) we conclude that 
\begin{equation}\label{constant}
\exp \left(-\frac{\pi\epsilon}{1 - \epsilon}\right)\int_{I}\omega (x) dx \leq \int_{I^*}\omega (x) dx \leq \exp \left(\frac{\pi\epsilon}{1 - \epsilon}\right)\int_{I}\omega (x) dx. 
\end{equation}
This completes the proof of the doubling condition. 
\end{proof}

\begin{remark}
We see from the above arguments that 
\begin{equation*}
\exp \left(-\frac{\pi\epsilon t}{1 - \epsilon}\right)\int_{I}\omega^t (x) dx \leq \int_{I^*}\omega^t (x) dx \leq \exp \left(\frac{\pi\epsilon t}{1 - \epsilon}\right)\int_{I}\omega^t (x) dx, 
\end{equation*}
which implies $\omega^t$ is also a doubling weight for any $0 \leq t \leq 1$. Since $\epsilon\in (0,1)$ may be arbitrarily small, we may moreover assume that the doubling constant of the weight $\omega^t$ is as close to $1$ as we like. 
\end{remark}

\section{Proof of Corollary 2}
A locally integrable function $\omega \geq 0$ on the real line $\mathbb R$ is called an $A_p$-weight for $1 < p < \infty$ if 
$$
\sup_I \left( \frac{1}{|I|} \int_I \omega (x) dx \right) \left( \frac{1}{|I|} \int_I \left(\frac{1}{\omega(x)}\right)^{\frac{1}{p-1}} dx\right)^{p-1} < \infty, 
$$
where the supremum is taken over all bounded intervals. It is known that $A_p \subset A_q$ if $p < q$ and $A_{\infty} = \cup_{p > 1}A_p$. A locally integrable function $\omega \geq 0$ on the real line $\mathbb R$ is called an $A_1$-weight, if there is a constant $C > 0$ such that for all bounded intervals $I$
$$
\omega_I \leq C\inf_I\omega, 
$$
or equivalently, there is a constant $C > 0$ such that 
$$
M\omega(x) \leq C\omega(x)
$$
almost everywhere on $\mathbb R$. If $\omega(x)$ satisfies $A_1$, then $\omega(x)$ satisfies $A_p$ for any $p > 1$. The following result by Coifman and Rochberg \cite{CR} establishes the relationship between $A_1$-weights and Hardy-Littlewood maximal functions.
\begin{proposition}\label{CoRo}
If $\nu$ is a locally finite signed Borel measure with $M\nu(x) < \infty$ almost everywhere, then $(M\nu)^t$ is an $A_1$-weight for any $0 \leq t < 1$.
\end{proposition}
As was observed in \cite{CR}, this construction yields essentially all the elements of $A_1$-weights and in fact essentially all of $A_1$-weights are obtained using only signed measures of the form $g(x)dx$. 

\begin{proof}[Proof of Corollary 2]
We come back to our constructions of the function $\tilde f$ and the weight function $\omega = M\tilde f$ in section 2. 
Set the measure $\nu(x) = \tilde f(x)dx$ so that $M\nu = \omega$.  
We conclude by Proposition \ref{CoRo}  that $|h_t'| = \omega^t$ is an $A_1$-weight, and thus $A_{\infty}$-weight for any $0 \leq t < 1$. On the other hand, it follows from Theorem 1 that $|h'| = \omega$ is not an $A_{\infty}$-weight.  Moreover, $|h_t'| = \omega^t$ is a doubling weight for any $0 \leq t \leq 1$. 

It is observed that
$$
\Vert \log|h_t'| \Vert_* = \Vert t\log|h'| \Vert_* \leq \Vert \log|h'| \Vert_* \leq C, 
$$
where $C > 0$ is some constant. 
By the equivalences of $(1)$, $(2)$ and $(3)$ in section 1,  we have $\log\Phi'_t \in \rm BMOA(\mathbb D)$ and $\log\Psi'_t \in \rm BMOA(\mathbb D^*)$ for any $0 \leq t < 1$, but $\log\Phi'_1\notin \rm BMOA(\mathbb D)$ and $\log\Psi'_1 \notin \rm BMOA(\mathbb D^*)$. 

It remains to show that $\Vert \log\Phi_t'\Vert_* \to \infty$ as $t \to 1$.   We suppose that there exists a subsequence $\{t_n\}$ converging to $1$ such that $\Vert \log\Phi_{t_n}'\Vert_* $ is bounded and we argue toward a contradiction.  Since $\mathrm {BMOA}(\mathbb D) = H^2 \cap \mathrm{BMO}(\mathbb S)$ is the dual of the classical space $H^1$, the sequence $\{ \log\Phi_{t_n}' \}$ has a weak-star convergent subsequence $\{ \log\Phi_{t_{n_k}}' \}$ converging to some function $\varphi$ in BMOA space in the following sense:
$$
\frac{1}{2\pi} \int_0^{2\pi} f(e^{i\theta})\overline{\log\Phi_{t_{n_k}}'(e^{i\theta})} d\theta \to  \frac{1}{2\pi} \int_0^{2\pi} f(e^{i\theta})\overline{\varphi(e^{i\theta})} d\theta
$$
as $k\to\infty$ 
for any $f \in H^1$. In particular, by taking $f \equiv 1$ we get $\log\Phi_{t_{n_k}}' \to \varphi$ almost everywhere on $\mathbb S$, and then by taking the Poisson integral we have 
$\log\Phi_{t_{n_k}}' \to \varphi$ almost everywhere on $\mathbb D$. On the other hand,  this subsequence $\{ \log\Phi_{t_{n_k}}' \}$ converges in the universal Teichm\"uller space in $\mathbb D$ to $\log\Phi_1'$, namely, $\Vert \log\Phi_{t_{n_k}}' - \log\Phi_1' \Vert_{\mathcal B} \to 0$. Thus, we conclude that $\log\Phi_1'  = \varphi \in \rm BMOA(\mathbb D)$. This leads to a contradiction. 
\end{proof}

\begin{question}
Let $\mathcal{C}=\{\log |h'|,\,|h'| \in A_\infty(\mathbb{S})\}$. It is an open convex subset of the real Banach space ${\rm BMO}_{\mathbb R}( \mathbb{S})$, the space of all real-valued BMO functions on $\mathbb S$. A paraphrase of our results (Theorem 1 and Corollary 2) is that there exists a quasisymmetric homeomorphism $h$ of $\mathbb S$ which is absolutely continuous with $\log\vert h'\vert\in \mathrm{BMO}_{\mathbb R}(\mathbb{S})$,   and moreover  
$\log\vert h'\vert\in \bar{\mathcal{C}}\setminus\mathcal C$, the boundary of $\mathcal C$ for the BMO topology.  We thus address the question:\\
Does there exist a quasisymmetric homeomorphism of $\mathbb S$ which is absolutely continuous with $\log\vert h'\vert\in \mathrm{BMO}_{\mathbb R}(\mathbb{S})$ such that $\log\vert h'\vert\notin \bar{\mathcal{C}}$?
\end{question}

\section{Proof of Theorem 3}
\begin{proof}[Proof of Theorem 3]
Recall that $|h'(e^{i\theta})| = \omega(\theta)$. We use $\omega(\theta)$ to denote $|h'(e^{i\theta})|$ for simplicity.  Here, $\omega$ is the weight function constructed in section 2. 
Since $\log \omega \in \rm BMO(\mathbb S)$, it is in particular integrable on the unit circle $\mathbb S$. If $z = re^{i\varphi}$, then the Poisson integral of $\log \omega$, 
$$
u(z) = \frac{1}{2\pi}\int_0^{2\pi}P_r(\varphi - \theta)\log \omega(\theta) d\theta = P_r\ast(\log \omega)(\varphi), 
$$
is harmonic on $\mathbb D$.  We let
$$
\log\hat\Phi' = u(z) + iv(z),  
$$
where $v(z)$ is the harmonic conjugate function of $u(z)$, normalized so that $v(0) = 0$. By $\log \omega \in \rm BMO(\mathbb S)$ again, we have that $v(z)$ has 
 nontangential limit almost everywhere on $\mathbb S$ which we denote by $b(\theta)$ , $b \in \rm BMO(\mathbb S)$, and thus $\log\hat\Phi' \in \rm BMOA(\mathbb D)$.  By the univalence criterion of Ahlfors-Weill \cite{AW}, $\hat\Phi$ is a conformal map onto a quasidisk and we call $\hat{\Gamma}$ its boundary whenever the doubling constant of $\omega$ is sufficiently small. By the Jensen inequality we have 
 \begin{equation*}
 \begin{split}
 |\hat\Phi'(z)| & = \exp(u(z)) \\
 &= \exp\left( \frac{1}{2\pi}\int_0^{2\pi}P_r(\varphi - \theta)\log \omega(\theta) d\theta \right)\\
 & \leq \frac{1}{2\pi}\int_0^{2\pi}P_r(\varphi - \theta) \omega(\theta) d\theta = P_r\ast\omega (\varphi).
 \end{split} 
 \end{equation*}
Since $\omega \in L^1$ on $\mathbb S$, we conclude that $P_r\ast \omega \in L^1$ on $\mathbb S$ for any $0 \leq r < 1$ and the sequence 
$\Vert P_r\ast \omega \Vert_1$ increases to $\Vert \omega \Vert_1$ as $r \to 1$. Then, 
$$
\Vert \hat\Phi' \Vert_{H^1} = \sup_{r}\frac{1}{2\pi}\int_0^{2\pi} |\hat\Phi'(re^{i\varphi})| d\varphi < \infty,
$$
which implies that $\hat\Phi' \in H^1(\mathbb D)$, actually $\hat\Phi'$ is an outer function,  and thus $\hat{\Gamma}$ is rectifiable. Moreover, since the boundary function $\omega$ of $\vert \hat\Phi' \vert$ is not an $A_{\infty}$-weight, $\hat{\Gamma}$ is not a chord-arc curve (\cite[Theorem 7.11]{Po}\cite{Zi}).
\end{proof}

\begin{question}
For any $0 \leq t < 1$, if we replace $\omega$ with $\omega^t$, set 
$$
\hat\Phi_t(z) = \hat\Phi_t(0) + \int_0^z (\hat\Phi')^t(\zeta) d\zeta
$$
and denote the curve $\partial \hat\Phi_t(\mathbb D)$ by $\hat{\Gamma}_t$. Then, by the same arguments as the above we have that $\hat{\Gamma}_t$ is a chord-arc curve since $\omega^t$ is an $A_{\infty}$-weight. According to these observations, the conformal map $\hat\Phi$ lies in the closure of chord-arc domain maps in the sense that $\hat\Phi_t$ is a conformal map onto a chord-arc domain for any $t \in [0, 1)$.  Must every map satisfying the conclusions of Theorem 3 be such? 
\end{question}

\begin{remark} 
Let $z(s)$ denote the arc-length parametrization of the chord-arc curve. Then, the set of all ${\rm arg} z'(s)$ forms an open subset of real-valued BMO functions (see \cite{Da}). 

We see from the above Question that the rectifiable curve $\hat\Gamma$ is on the boundary of  the closure of the space of chord-arc curves. Set 
$$
z(s) = z(0) + \int_0^s e^{i\beta (x)} dx
$$
is an arc-length parametrization of $\hat\Gamma$. 
Noting that the curve $\hat\Gamma$ has a parametrization $\gamma$ such that $\gamma'(t)=\omega(t)e^{ib(t)}$, we conclude that 
$$
z'(s)=  e^{i\beta (s)} = e^{ib\circ \alpha(s)}
$$
where $\alpha(s)$ is the inverse of the function
$$
s(t) = \int_0^t |\gamma'(x)| dx = \int_0^t \omega(x) dx. 
$$
Recall that  $b \in \rm BMO(\mathbb S)$  and $\omega$ is a positive $L^1$ function, but not an $A_{\infty}$- weight.  Thus,  we can not conclude that $\beta$ is a BMO function,  nor can we conclude that $z'(s)$ is not of this form!
\end{remark}

\end{document}